\documentclass{siamart}
\usepackage{makecell}
\usepackage{mathrsfs}
\usepackage{circuitikz}
\usepackage{amssymb}
\usepackage{colortbl}

\newsiamthm{assumption}{Assumption}
\newsiamthm{remark}{Remark} 

\title{Coercivity and Local Convergence of Physical Learning in Linear Circuits}
\author{Joshua A. McGinnis\footnotemark[1]
\and Xinbo Li\footnotemark[2]
\and Yoichiro Mori\footnotemark[2]\footnotemark[3]}
\headers{Coercivity and Local Convergence of Physical Learning}{J. A. McGinnis, X. Li, and Y. Mori}

\begin{document}
\maketitle

\renewcommand{\thefootnote}{\fnsymbol{footnote}}
\footnotetext[1]{Department of Mathematics, University of Pennsylvania, Philadelphia. Corresponding author (\email{jam887@sas.upenn.edu}).}
\footnotetext[2]{Department of Mathematics, University of Pennsylvania, Philadelphia.}
\footnotetext[3]{Department of Biology, University of Pennsylvania, Philadelphia.}

\begin{abstract}
Physical learning methods train physical networks to perform computational tasks using only local update rules, exploiting the physics of the system to handle the global transfer of information. We provide the first local convergence analysis of three such methods---Equilibrium Propagation (EP), Coupled Learning (CL), and a new method we call Adjoint Coupled Learning (AL)---for linear circuits, in the limit of small-nudging for both discrete and continuous time. EP and AL perform gradient descent on a natural loss function, while CL follows modified dynamics with an additional cubic correction. Assuming the existence of a solution, we identify a \emph{coercivity condition}, expressed as a rank condition on a matrix built from the network's incidence structure, under which the training loss decays exponentially and the parameters converge to the solution manifold. We show that coercivity can fail by exhibiting a kite circuit in which a symmetry causes the coercivity constant to degenerate on the solution manifold, but prove using Sard's theorem that such degeneracies are non-generic: coercivity holds at every point of the solution manifold for almost every choice of desired output.
\end{abstract}



\section{Introduction}

Physical learning refers to a class of training methods in which a physical network is trained to perform a computational task, and the only operation required at each network parameter is a local update rule. Unlike backpropagation, which requires propagating error signals from the output back through the entire network \cite{scellier2017equilibrium, kendall2020training}, physical learning methods exploit the physics of the system itself to extract the necessary update rules. In digital machines, backpropagation requires a global computation and therefore a central processor. In physical ``learning machines'' \cite{stern2021learning}, physics handles the global transfer of information without a central processor, through equilibrium dynamics. This is conjectured to lead to enhanced energy efficiencies at large scale \cite{kendall2020training, markovic2020physics}. Physical learning is therefore a candidate for large-scale, energy-efficient machine learning \cite{markovic2020physics}. Moreover, the decentralized nature of physical learning makes it a more plausible candidate for how biological systems learn, as it requires only local operations and no centralized computation \cite{stern2023learning}.

Recent experiments have demonstrated that physical learning can be implemented in real analog devices. Dillavou et al.\ \cite{dillavou2022demonstration} trained networks of adjustable resistors using Coupled Learning, Stern et al.\ \cite{stern2020supervised} demonstrated learning in elastic networks, and Kendall et al.\ \cite{kendall2020training} showed that Equilibrium Propagation can train analog circuits. These results build on a longer tradition of energy-based learning in neural networks, going back to Boltzmann machines \cite{hinton2002training} and contrastive Hebbian learning in Hopfield networks \cite{movellan1991contrastive}. As physical learning methods move toward practical deployment at scale \cite{laborieux2021scaling}, the question of convergence becomes pressing: under what conditions is a physical system guaranteed to learn, and what structural features of the network can prevent convergence?

The two most prominent physical learning methods are \emph{Equilibrium Propagation} (EP), introduced by Scellier and Bengio \cite{scellier2017equilibrium} and extended to a model-free setting in \cite{scellier2022agnostic}, and \emph{Coupled Learning} (CL), introduced by Stern and collaborators \cite{stern2021learning}. Both operate by comparing two equilibrium states: a \emph{free state}, in which the system relaxes under its natural dynamics, and a \emph{nudged state}, in which the output is gently pushed toward a target. For reversible (Hamiltonian) systems, time-reversal of the dynamics provides an alternative to equilibrium-based methods \cite{lopezpastor2023selflearning}. The parameter update is then proportional to the difference in local quantities between these two states. EP uses the error to apply a small force to the energy, while CL uses the error to constrain the output toward the target. Constraining the output is a simpler physical operation than applying a force proportional to the error, which requires an active source \cite{stern2021learning}. In this paper, we introduce a third method, \emph{Adjoint Coupled Learning} (AL), which constrains the output to the target and nudges using the resulting constraint force. AL preserves the physical simplicity of CL while, as we show below, recovering the gradient flow structure of EP.

The gradient equivalence between contrastive update rules and backpropagation was first established by Xie and Seung \cite{xie2003equivalence} for layered networks, and generalized to arbitrary energy-based systems by Scellier and Bengio \cite{scellier2017equilibrium}, who showed that EP performs gradient descent on a natural loss function in the small-nudge limit. The analogous calculation for CL reveals that it is \emph{not} gradient flow of any apparent loss function, and the CL update computes only part of the gradient of a modified, parameter-dependent loss \cite{stern2021learning}. Despite this understanding, rigorous convergence results have been scarce. The existing literature demonstrates the gradient structure but stops short of proving that the resulting dynamics actually drive the system to a trained state. A recent exception is \cite{huijzer2026convergence}, which proves global convergence of contrastive learning for linear resistor networks without hidden nodes, where the cost function is convex; our approach is local but applies to networks with hidden nodes, where convexity is not available.

In this paper, we provide the first local convergence analysis for EP, CL, and AL in both the small-nudge, small-learning-rate (continuous time) limit, and for sufficiently small-nudge, discrete time setting. Our main results (Theorem~\ref{thm:local_convergence}, Theorem~\ref{thm:discrete_EP}, and Theorem~\ref{thm:discrete_CL}) state that if the system starts sufficiently close to a solution (a parameter configuration achieving the desired output) and a \emph{coercivity condition} is satisfied, then the loss decays exponentially and the parameters converge to the solution manifold. The proof uses a bootstrap argument, i.e.\@ the Implicit Function Theorem provides local estimates, coercivity ensures the loss decreases exponentially, and the two together guarantee that the system never leaves the region where the estimates hold.

The coercivity condition is the key requirement for local exponential convergence to hold. For EP, it reduces to a rank condition on a matrix built from the circuit's active-incidence structure and the constrained Laplacian (Theorem~\ref{thm:EP_coercivity}). For CL, the analogous condition involves a slightly different matrix, and the loss dynamics include an additional cubic correction term that must be controlled (Theorem~\ref{thm:CL_coercivity}). We show that coercivity can fail for both: we exhibit a ``kite'' circuit in which a symmetry forces the free and clamped voltage drops to have disjoint support, making the dissipative term degenerate. However, we show using Sard's theorem that coercivity holds at every point of the solution manifold for almost every choice of desired output (Proposition~\ref{prop:genericity}).

The paper is organized as follows. Section~\ref{sec:primer} is a self-contained primer that introduces physical learning through a single concrete resistor network, assuming no prior background; the reader familiar with the subject may skip it. Section~\ref{Mathematical Set up} introduces the general mathematical framework for physical learning: energy minimization, free and nudged states, and the parameter update rule in the small-nudge limit. Section~\ref{EP C Laws} computes the loss dynamics for EP, CL, and AL, establishing the gradient flow structure for EP and AL and the modified dynamics for CL. Section~\ref{sec:coercivity} develops the coercivity estimates for all three methods. Section~\ref{sec:kite} presents the kite circuit counterexample showing that coercivity can fail, establishes the genericity of coercivity via Sard's theorem, and includes numerical experiments. Section~\ref{sec:convergence} contains the main continuous-time convergence theorem and its proof. Section~\ref{sec:discrete} extends the analysis to the discrete update rule at finite learning rate, establishing geometric decay of the loss via a descent lemma combined with a Polyak--\L{}ojasiewicz inequality coming from coercivity.
\section{A Primer on Physical Learning}
\label{sec:primer}
Physical learning asks: can a physical network be trained, \emph{using local update rules and its own physics}, to produce desired outputs from given inputs? Who or what does the training is left deliberately open; a physical system may be trained by a human, by an environment, or even by itself \cite{stern2023learning}. We say a system has ``learned'' when its elements produce coherent behavior, responsive to the inputs, that is beneficial or desired. In this section we adopt the perspective of an engineer training a circuit. The algorithms we study have not been observed in nature in their exact form, but because they are local, they capture some aspect of how a physical system could learn from its environment. The language of electrical circuits is simply our preference; one could equally speak of spring networks, chemical networks, or fluid networks.

Consider the electrical network in Figure~\ref{fig:primer_circuit}.
It consists of nodes joined by edges, and each edge contains a resistor whose conductance controls the flow of electricity across it. We impose inputs as voltage constraints at a few nodes: in our example, node $a$ is held at voltage $1$ and node $b$ is grounded (held at voltage $0$).
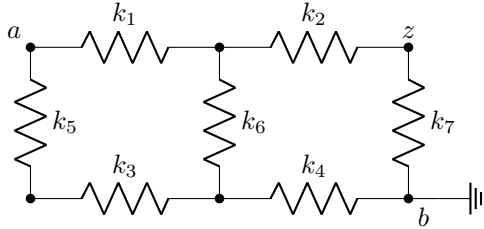
\begin{figure}[ht]
  \centering
  \begin{circuitikz}
  \node[circ] (a) at (0,2)   {};
  \node[circ] (c) at (2.5,2) {};
  \node[circ] (z) at (5,2)   {};
  \node[circ] (d) at (0,0)   {};
  \node[circ] (e) at (2.5,0) {};
  \node[circ] (b) at (5,0)   {};

  \node[above left]  at (a) {$a$};
  \node[above]       at (z) {$z$};
  \node[below right] at (b) {$b$};

  \draw (a) to[R=$k_1$] (c);
  \draw (c) to[R=$k_2$] (z);
  \draw (d) to[R=$k_3$] (e);
  \draw (e) to[R=$k_4$] (b);

  \draw (a) to[R=$k_5$] (d);
  \draw (c) to[R=$k_6$] (e);
  \draw (z) to[R=$k_7$] (b);

  \draw (b) -- ++(0.4,0) node[ground, rotate=90] {};
\end{circuitikz}
  \caption{A resistor network with $N = 6$ nodes and $M = 7$ edges. The input nodes are $a$, held at voltage $1$, and $b$, held at voltage $0$ (ground); the output node is $z$. Each edge $e$ carries a resistor with conductance $k_e$.}
  \label{fig:primer_circuit}
\end{figure}
We wish the network to attain a desired voltage at one or more designated output nodes. In our example, we ask that the output node $z$ attain voltage $w = 1/2$.

To achieve this we adjust the conductances of the resistors, though not in an arbitrary way; we come to this shortly. Let $k \in \mathbb{R}^{M}$ denote the vector of the $M$ conductances. Because the circuit is linear, the output voltage depends linearly on the inputs once $k$ is fixed, and training amounts to adjusting $k$ until the output matches the target. One could find the node voltages directly from Kirchhoff's laws, but more revealingly they arise from a constrained energy minimization, the energy being the one the physical system itself minimizes.

For a linear circuit this energy is the dissipated power,
\[G(x ; k) = \sum_{e \in \{\text{network edges}\}}k_e(\Delta_e x)^2,\]
where $x \in \mathbb{R}^N$ collects the node voltages, $k_e$ is the conductance of edge $e$, and $\Delta_e x$ is the voltage drop across edge $e$ (with an arbitrary but fixed orientation). Holding the inputs fixed, the circuit settles into its \emph{free state} $x_0$: the node voltages that minimize the power subject to the input constraints (node $a$ at voltage $1$, node $b$ grounded). We write $x_0 = x_0(k)$ to emphasize that the free state depends on all the conductances; it is a global object, in that every node voltage depends on the entire network. We call the resulting voltage at node $z$ the \emph{free output}.

On a computer, one could find a suitable $k$ by gradient descent on the squared error between the free output and its target. This requires computing the free state explicitly, which in turn amounts to inverting a constrained graph Laplacian. Such an inversion is ``non-local'': no sub-block of the inverse can be formed without reference to the rest of the matrix. It is this global transfer of information that is posited to be a major source of inefficiency \cite{stern2021learning, dillavou2022demonstration}.

We therefore seek an algorithm in which the physics of the system itself performs the optimization. The central idea is \emph{nudging}: we perturb the system at the output node, either by injecting a current or by enforcing a voltage constraint. Focusing on the voltage version, called Coupled Learning, we form the \emph{nudged state} $x_\eta$ by re-solving the same minimization with node $z$ no longer free but clamped a fraction of the way from its free output toward the target, at a convex combination of the two governed by a small nudging parameter $\eta > 0$. At $\eta = 0$ the output is left free; as $\eta$ increases it is pulled toward the target. Nudging with a current instead of a voltage yields Equilibrium Propagation.

Producing the free and nudged states requires external controllers only at the input and output nodes. Updating the conductances requires, at each edge, a controller that observes the voltages at the two nodes framing that edge. For the linear circuit the update is
\[\Delta k_e = -\dfrac{\tau}{\eta}\left(\left(\Delta_e x_\eta\right)^2 - \left(\Delta_e x_0\right)^2 \right),\]
where $\tau > 0$ is a small learning rate.

To summarize, both $x_0$ and $x_\eta$ depend on the entire network, but the physics computes them for us. The engineer's only intervention is to measure the voltage drops $\Delta_e x_0$ and $\Delta_e x_\eta$ at each edge and to update that edge using only those two quantities. No system other than the physical one is needed to transfer information from one edge to another. In the next section we introduce the form of the update that generalizes to other kinds of physical systems.

\section{Mathematical Set-Up}
\label{Mathematical Set up}

We now proceed to the general framework and introduce additional structure as necessary.
Let
\[
    G(x;k): \mathbb{R}^{N} \times \mathbb{R}^{M} \to \mathbb{R}
\]
be a smooth function, convex in $x$ for each admissible $k$. Here $x \in \mathbb{R}^N$ are the \textit{state variables} (e.g., node voltages) and $k \in \mathbb{R}^M$ are the \textit{trainable parameters} (e.g., edge conductances).

Let $v \in \mathbb{R}^I$ denote an input. Generally $I\ll N.$ We encode the input through a matrix $P \in \mathbb{R}^{N \times I}$, and we assume that in the \emph{free state}, the physical system minimizes its energy subject to enforcing this input on state variables:
\begin{equation}
    x_0 
    = \arg\min_{x} G(x;k)
    \quad\text{subject to} \quad 
    P^\top x = v.
    \label{eq:free-state}
\end{equation}
We assume that this constrained minimization problem has a unique solution for each admissible $k$; equivalently, $G(\cdot\,;k)$ is strictly convex on the affine subspace $\{x : P^\top x = v\}$. During this stage, the parameters $k$ are held fixed.
The notation $x_0$ will become meaningful once we contrast it with the nudged/clamped states.

An example of a matrix $P$ comes from the example given in Section~\ref{sec:primer}. If nodes $a$ and $b$ are indexed as the first two nodes, then $P$ is given by 
\[P = \begin{bmatrix}
    1 & 0 \\
    0 & 1 \\ 
    0 & 0 \\
    0 & 0 \\
    0 & 0 \\
    0 & 0
\end{bmatrix}, \qquad v = \begin{bmatrix}
    1 \\0
\end{bmatrix}.\]

We encode the desired output through another matrix $Q \in \mathbb{R}^{N \times O}$ with $O \ll N$. The $Q$ coming from the example given in Section~\ref{sec:primer} is a $6 \times 1$ matrix, which is all zeros except for a one in its last entry if $z$ is the last node (i.e.\@, $Q$ can be identified as the natural basis vector $e_6$). In that example $w$ is $1 \times 1$ i.e. the scalar quantity $w=1/2.$

We assume that the input variables and the output variables are disjoint. Given a target output $w \in \mathbb{R}^O$, we define the \emph{state error} as
\[
    r := Q^\top x_0 - w.
\]
The goal of physical learning is to adjust $k$ in order to decrease $r$.

Equilibrium Propagation and Coupled Learning modify the free state by clamping (forcing or constraining respectively) the output sites. They use a small nudging parameter $\eta > 0$ to perturb these sites away from the free state and toward the desired objective. The two methods differ in \emph{how} the nudge is applied.

\paragraph{Equilibrium Propagation (EP)}
EP introduces a small \textit{force} proportional to the residual:
\begin{equation}
    x_\eta 
    = \arg\min_{x} 
    \big(
        G(x;k) - \eta\, \langle Qr, x\rangle
    \big)
    \quad\text{subject to}\quad
    P^\top x = v.
    \label{eq:EP}
\end{equation}

\paragraph{Coupled Learning (CL)}
Coupled Learning introduces the nudge as a constraint rather than as a force:
\begin{equation}
    x_\eta
    = \arg\min_{x} G(x;k)
    \quad\text{subject to}\quad
    P^\top x = v, 
    \qquad 
    Q^\top x =(1+\eta)Q^Tx_0 -\eta w .
    \label{eq:CL}
\end{equation}

In both cases, $x_\eta \to x_0$ as $\eta \to 0$.

\paragraph{Learning Update}

Let $k_i^{(n)}$ denote the $i$th parameter after the $n$th training iteration.
The physically implemented learning rule is of the form
\[
    k_i^{(n+1)} 
    = k_i^{(n)}
      + \dfrac{\tau}{\eta}\big(
      \partial_{k_i}G(x_\eta(k^{(n)});k^{(n)})
           - \partial_{k_i}G(x_0(k^{(n)});k^{(n)}) \big),
\]
where $\tau>0$ is the learning rate.
\begin{remark}
    The notation $\partial_{k_i}G$ and $\partial_{k_i}\nabla_xG$ should always be interpreted as the derivative of the functions with respect to their $k_i$ argument. On the other hand, if other arguments of $G$ depend functionally on $k_i$ and we want the derivative to also act on these, we could write $\partial_{k_i}(G)$. For example
    \[
    \partial_{k_i}(G(x_0(k);k)) = \nabla_{x}G(x_0(k);k) \cdot \partial_{k_i}x_0(k) + \partial_{k_i}G(x_0(k);k).
\]
\end{remark}

\paragraph{Adjoint Coupled Learning (AL)}
Both CL and EP use the state error to nudge. In EP the error is used to supply a small forcing. In CL, the error is used to nudge the state variables themselves by enforcing an additional constraint. 

Here we introduce \textit{Adjoint Coupled Learning} (AL) as an alternative. First, we introduce the \textit{reference state}
\begin{equation}
    x_0
    = \arg\min_{x} G(x;k)
    \quad\text{subject to}\quad
    P^\top x_0 = v, 
    \qquad 
    Q^\top x_0 = w .
    \label{eq:AL}
\end{equation}

\begin{remark}
    This reference state, $x_0$, is not the free state defined in EP and CL. However, it plays the same role as the free state in the learning update. Therefore, we denote it by $x_0$. The goal of AL is to make the constraint $Q^\top x_0=w$ ``redundant" (see below), in which case $x_0$ becomes the free state with the desired output.
\end{remark}

Next, we define the Lagrangian of this system:
\begin{equation}
    \mathcal{L}(x_0,\lambda_0,\mu_0 ; k):= G(x_0;k)+\langle \lambda_0, P^\top x_0-v\rangle+\langle \mu_0, Q^\top x_0-w\rangle
    \label{eq:AL_mu_0}
\end{equation}
with $\lambda_0$ and $\mu_0$ the Lagrange multipliers (e.g., currents). We call $\mu_0$ the \textit{adjoint error}, which we use to nudge. The nudged problem is 
\begin{equation}
    x_\eta
    = \arg\min_{x} G(x;k)
    \quad\text{subject to}\quad
    P^\top x_\eta = v, 
    \qquad 
    Q^\top x_\eta = w +\eta \mu_0.
    \label{eq:AL_nudge}
\end{equation}
In the sequel, we see that the norm of $\mu_0$ serves as a loss. In the small nudging limit, we recover gradient flow of this loss.  When $\mu_0 =0$, the reference state $x_0$ becomes exactly the free state of CL and EP and it satisfies the objective $Q^\top x_0= w$ without needing to enforce it as a constraint.

\paragraph{Learning Dynamics}
We consider the mathematically idealized continuous time training. Sending $\tau \to 0$ and rescaling time gives the continuous-time limit
\begin{equation}
    \dot{k}_i
    =
    \dfrac{1}{\eta}\,\big(\partial_{k_i}
        G(x_\eta;k) - \partial_{k_i}G(x_0;k)
    \big).
    \label{eq:continuous-dynamics}
\end{equation}
In this paper, we consider the limit of small nudging $\eta \to 0$
\begin{equation}
    \dot{k}_i
    =
   \partial_{k_i}
     \left\langle \nabla_x G(x_0;k),\, \dfrac{d}{d\eta}x_\eta\bigg|_{\eta=0} \right\rangle .
    \label{eq:leading-order-dynamics}
\end{equation}
Equation 
\eqref{eq:leading-order-dynamics} is the fundamental evolution equation for the remainder of the paper.

\section{Gradient Flow and Not}
\label{EP C Laws}
Next, we look more closely at the dynamics of EP, CL, and AL. It has already been shown before that EP performs gradient flow \cite{scellier2017equilibrium}. In \cite{stern2021learning}, the authors compare the calculation with the analogous calculation for CL and see that CL is not quite gradient flow. We redo both these calculations for completeness, expanding on the calculation for CL in a way that is helpful for our local convergence argument in the sequel. Finally, we show that AL also performs gradient flow but for a different loss.

Starting with EP and CL, we define the residual, as before, as 
\[ r(k) := Q^\top x_0(k) - w\] 
and the loss relevant to EP and  CL as 
\[
    \Phi(k) := \frac{1}{2}\|r(k)\|^{2}.
\]
To disambiguate, we introduce the loss for AL later.

\paragraph{Gradient of the Loss}
By convexity of $G(\cdot \,;k)$, \eqref{eq:free-state} is equivalent to the KKT system
\begin{align}
\nabla_x G(x_0; k) + P\lambda = 0,\qquad P^\top x_0 = v, \label{eq:kkt}
\end{align}
for some Lagrange multiplier $\lambda_0\in\mathbb{R}^I$.

Taking the total derivative of \eqref{eq:kkt} with respect to $k_i$ gives
\begin{align}
\begin{bmatrix}
H & P\\
P^\top & 0_{I \times I}
\end{bmatrix}
\begin{bmatrix}
\partial_{k_i}x_0\\[2pt]
\partial_{k_i}\lambda_0
\end{bmatrix}
= -
\begin{bmatrix}
\partial_{k_i}\nabla_x G\\
0_{I \times 1}
\end{bmatrix},\qquad
H := \nabla_x^2 G(x_0;k).\label{eq:lin-kkt}
\end{align}
Let
\begin{align}
S := \begin{bmatrix} \mathrm{Id}_{N\times N}\\ 0_{I\times N}\end{bmatrix},\qquad
R := S^\top
\begin{bmatrix}
H & P\\
P^\top & 0_{I \times I}
\end{bmatrix}^{-1}\!
S.\label{eq:R}
\end{align}
By the strict convexity of $G(\cdot\,;k)$ on the constraint surface, $H$ is positive definite on $\ker P^\top$, which ensures the bordered Hessian is invertible. Hence
\begin{align}
\partial_{k_i}x_0 = - R\,\partial_{k_i}\nabla_x G.\label{eq:phi-sens}
\end{align}
and hence
\begin{equation}\partial_{k_i} \Phi(k) = \langle Qr,  \partial_{k_i}x_0(k)\rangle =-\langle RQr,  \partial_{k_i}\nabla_x G\rangle . 
\label{eq:dynamics-grad}
\end{equation}

\paragraph{EP Executes Gradient Flow}

The KKT conditions for the EP minimization problem \eqref{eq:EP} read
\begin{align}
\nabla_x G(x_\eta) + P\lambda_\eta = \eta Qr,\qquad P^\top x_\eta = v.\label{eq:kkt-eta-EP}
\end{align}
Differentiating in $\eta$ at $\eta=0$ yields
\begin{align}
\begin{bmatrix}
H & P\\
P^\top & 0
\end{bmatrix}
\begin{bmatrix}
y\\
\alpha
\end{bmatrix}
=
\begin{bmatrix}
Qr\\
0
\end{bmatrix}, \quad \alpha :=\left.\dfrac{d}{d\eta}\lambda_\eta\right|_{\eta=0} ,\quad
y := \left. \frac{d}{d\eta}x_\eta\right |_{\eta=0}\label{eq:psi}
\end{align}
We can solve $y = RQr$ which can be regarded as the minimizer for an energy that is the quadratic approximation of $G(\cdot; k)$. Hence
\begin{align}
\dot{k}_i=\lim_{\eta\to 0}\frac{1}{\eta}\Big(\partial_{k_i}G(x_\eta;k)-\partial_{k_i}G(x_0;k)\Big)
= \big\langle y, \nabla_x \partial_{k_i}G\big\rangle
= \big\langle R Qr,\partial_{k_i}\nabla_x G\big\rangle\label{eq:dynamics-EP},
\end{align}
which is the negative of $\partial_{k_i}\Phi$. The error of the state variables, $r$, serves as forcing signal in a quadratic minimization problem. In light of \eqref{eq:dynamics-grad}, the dynamics of the loss function are given by
\begin{equation}
    \dot{\Phi} = - \sum_{i=1}^M\big\langle R Q r,\partial_{k_i}\nabla_x G\big\rangle^2=- \sum_{i=1}^M\dot{k}_i^2.
\end{equation}
Therefore the loss must decrease when the $k$ are not at a fixed point.

\paragraph{Not Quite Gradient Flow for CL}
Recall the nudged energy minimization problem \eqref{eq:CL} with minimizer $x_{\eta}$. The KKT conditions read
\begin{align}
\nabla_x G(x_\eta;k) + P\lambda_\eta +Q\mu_\eta = 0,\quad
P^\top x_\eta = v, \quad
Q^\top x_\eta = Q^\top x_0 + \eta r,
\label{eq:kkt-eta-CL}
\end{align}
with Lagrange multipliers $\lambda_\eta \in \mathbb{R}^I$ and $\mu_\eta \in \mathbb{R}^{O}.$
Differentiating in $\eta$ at $\eta=0$ yields
\begin{align}
\begin{bmatrix}
H & P & Q\\
P^\top & 0 & 0 \\
Q^\top & 0 & 0
\end{bmatrix}
\begin{bmatrix}
y\\[2pt]
\alpha\\
\beta
\end{bmatrix}
=
\begin{bmatrix}
0\\
0\\
r
\end{bmatrix},
\quad
\alpha :=\left.\dfrac{d}{d\eta}\lambda_\eta\right|_{\eta=0},
\quad
\beta :=\left.\dfrac{d}{d\eta}\mu_\eta\right|_{\eta=0},
\quad
y:= \left. \frac{d}{d\eta}x_\eta\right |_{\eta=0}.
\label{eq:psi2}
\end{align}
Solving, we find
\[
y =R Q\beta= R Q (Q^\top RQ)^{-1}r.
\]
Hence
\begin{align}
\dot{k}_i= \lim_{\eta\to 0}\frac{1}{\eta}\Big(\partial_{k_i}G(x_\eta;k)-\partial_{k_i}G(x_0;k)\Big)
= \big\langle y, \nabla_x \partial_{k_i}G\big\rangle
= \big\langle R Q \beta,\partial_{k_i}\nabla_x G\big\rangle.
\label{eq:dynamics-CL}
\end{align}

Define $D:=(Q^\top RQ)^{-1}$. Since $R$ is positive semi-definite with $\ker R = \operatorname{range}(P)$, the disjoint input/output assumption ensures that $Q^\top RQ$ is positive definite, and hence $D$ is well-defined and positive definite.

\begin{remark}
$D$ is a discrete Dirichlet-to-Neumann map: it converts a voltage boundary condition into the equivalent current input. CL nudges with the voltage error $r$, but solving the linearized KKT system, this is equivalent to applying a current $\beta = Dr$. Here $\beta$ is the linearized adjoint error---that is, $\beta$ equals the adjoint error $\mu_0$ exactly when $G$ is quadratic in $x$. In either interpretation, CL matches error and nudge of the same type (voltage to voltage, or equivalently current to current), while EP crosses the duality by applying the voltage error as a current.
\end{remark} 

Motivated by \eqref{eq:dynamics-CL}, consider the weighted loss
\begin{equation}
    \Phi_D(k):=\frac12\langle r, Dr\rangle =\frac12\langle r, \beta \rangle.
\end{equation}
Since $D=D(k)$ depends on time through $k(t)$, we have
\begin{equation}
\dot D
= \sum_{i=1}^M (\partial_{k_i}D)\,\dot k_i
= \sum_{i=1}^M (\partial_{k_i}D)\,\big\langle R Q Dr,~\partial_{k_i}\nabla_x G\big\rangle,
\label{eq:Ddot}
\end{equation}
where in the second equality we used the CL learning rule
$\dot k_i=\langle RQDr,\partial_{k_i}\nabla_xG\rangle$.

Differentiating $\Phi_D=\frac12 \langle r, D r\rangle$ gives
\begin{equation}
\dot{\Phi}_D
= - \sum_{i=1}^M\big\langle R Q Dr,\partial_{k_i}\nabla_x G\big\rangle^2
+\frac12\big\langle r,\dot D\,r\big\rangle.
\label{eq:PhiD_dot}
\end{equation}
The first term is the gradient-flow dissipation term. The second term arises solely from the $k$-dependence of the weight $D$; using \eqref{eq:Ddot}, it is cubic in $r$ in the sense that it contains two factors of $r$ from $\langle r,(\partial_{k_i}D)r\rangle$ and one factor of $r$ from $\langle RQDr,\partial_{k_i}\nabla_xG\rangle$. Thus, for sufficiently small $r$, the leading-order behavior is dominated by the dissipative term. In the sequel, we make this precise.
\paragraph{Gradient of the Adjoint Loss}
In what follows $x_0, \lambda_0,$ and $\mu_0$ reflect the state and adjoint variables of the reference state and not the free state. We define the \textit{adjoint loss}  as 
\[
\Phi^*(k) := \tfrac12\|\mu_0(k)\|^2
\]
where $\mu_0$ is the Lagrange multiplier found in \eqref{eq:AL_mu_0}. The reference state $x_0$ for AL solves 
\begin{equation}
\label{eq:AL_KKT_ref}
\nabla_x G(x_0;k) + P\lambda_0 + Q\mu_0 = 0,
\qquad
P^\top x_0 = v,
\qquad
Q^\top x_0 = w .
\end{equation}
We have that
\[
\partial_{k_i} \Phi^*(k)
=
\langle \mu_0, \partial_{k_i}\mu_0\rangle.
\]
We take the total derivative of the reference KKT conditions \eqref{eq:AL_KKT_ref} with respect to the parameter $k_i$.
This yields
\begin{equation}
\label{eq:AL_kdiff}
\begin{bmatrix}
H_{\rm AL} & P & Q\\
P^\top & 0 & 0\\
Q^\top & 0 & 0
\end{bmatrix}
\begin{bmatrix}
\partial_{k_i} x_0\\
\partial_{k_i} \lambda_0\\
\partial_{k_i} \mu_0
\end{bmatrix}
=
-
\begin{bmatrix}
\partial_{k_i}\nabla_x G(x_0;k)\\
0\\
0
\end{bmatrix},
\end{equation}
where $H_{\rm AL}:=\nabla^2_{x}G(x_0;k)$.
\begin{remark}
The difference between $H_{\rm AL} $ and $H$ is that the  first is evaluated at the reference state of AL whereas the second is evaluated at the free state of CL or EP. These are not the same in general until the objective is satisfied.
\end{remark}

Consider separately the KKT system given by 
\begin{equation}
\label{eq:AL_eta}
\begin{bmatrix}
H_{\rm AL} & P & Q\\
P^\top & 0 & 0\\
Q^\top & 0 & 0
\end{bmatrix}
\begin{bmatrix}
 y\\
\alpha \\
\beta
\end{bmatrix}
=
\begin{bmatrix}
 0\\
0\\
\mu_0
\end{bmatrix}.
\end{equation}
Taking the inner product of $y$ with the first equation of the KKT system \eqref{eq:AL_kdiff}, we have that 
\begin{equation*}
    \langle y,H_{\rm AL}\partial_{k_i} x_0 \rangle+ \langle y,P\partial_{k_i} \lambda_0 \rangle+ \langle y,Q\partial_{k_i} \mu_0 \rangle  = -\langle y,\partial_{k_i} \nabla_xG(x_0;k)\rangle
  \end{equation*}
Using that $P^\top y=0, Q^\top y = \mu_0$, and $H_{\rm AL}y = -P\alpha-Q\beta$, we find that:
  \begin{equation*}
       \langle -P\alpha -Q\beta,  \partial_{k_i} x_0 \rangle+ \langle \mu_0,\partial_{k_i} \mu_0 \rangle  = -\langle y,\partial_{k_i} \nabla_xG(x_0;k)\rangle.
\end{equation*}
Finally we use that $P^\top \partial_{k_i}x_0 = Q^\top \partial_{k_i}x_0 =0$ to get
\begin{equation}
      \partial_{k_i}\Phi^*=   \langle \mu_0,\partial_{k_i} \mu_0 \rangle  = -\langle y,\partial_{k_i} \nabla_xG(x_0;k)\rangle.
      \label{eq:grad_adjoint_loss}
\end{equation}

\paragraph{AL Executes Gradient Flow}
The nudged state $x_\eta$ satisfies
\begin{equation}
\label{eq:AL_KKT_eta}
\nabla_x G(x_\eta;k) + P\lambda_\eta + Q\mu_\eta = 0,
\qquad
P^\top x_\eta = v,
\qquad
Q^\top x_\eta = w + \eta\mu_0 .
\end{equation}
Differentiating \eqref{eq:AL_KKT_eta} with respect to $\eta$ at $\eta=0$ yields the linearized
KKT system found in \eqref{eq:AL_eta} with
\begin{equation}
\label{eq:AL_linKKT}
y := \left.\frac{d}{d\eta}x_\eta\right|_{\eta=0},\alpha := \left.\frac{d}{d\eta}\lambda_\eta\right|_{\eta=0},\beta := \left.\frac{d}{d\eta}\mu_\eta\right|_{\eta=0}.
\end{equation}
Solving \eqref{eq:AL_linKKT} gives
\begin{equation}
\label{eq:AL_y}
y = R_{AL}Q D_{\rm AL}\,\mu_0.
\end{equation}
where $R_{\rm AL}$ is defined analogously to $R$ in \eqref{eq:R} but using $H_{\rm AL}$ instead of $H$. We also have $D_{\rm AL} = (Q^\top R_{\rm AL} Q)^{-1}$. We obtain
\begin{equation}
\begin{aligned}
\label{eq:AL_limit}
\dot{k}_i=\lim_{\eta\to0}
\frac{1}{\eta}
\big(
\partial_{k_i}G(x_\eta;k)
-
\partial_{k_i}G(x_0;k)
\big)
&=\big\langle y,\;\partial_{k_i}\nabla_x G(x_0;k)\big\rangle
\\&=
\big\langle R_{\rm AL}Q D_{\rm AL}\mu_0,\;\partial_{k_i}\nabla_x G(x_0;k)\big\rangle ,
\end{aligned}
\end{equation}
 Comparing with \eqref{eq:grad_adjoint_loss}, we see this is exactly the negative gradient of the loss, so 
 \begin{equation}
    \dot{\Phi}^* = - \sum_{i=1}^M\big\langle R_{\rm AL} Q D_{\rm AL}\mu_0,\partial_{k_i}\nabla_x G\big\rangle^2=- \sum_{i=1}^M\dot{k}_i^2.
\end{equation}

\paragraph{Reinterpreting AL as EP}
Finally, we note that AL can be formally associated with the augmented energy $\tilde{G}(\tilde{x},k):=G(x,k) +\mu (Q^\top x-w)$, where $\tilde{x} := (x,\mu)$. Although $\tilde{G}$ is not convex in $\tilde{x}$, the constrained critical point of $\tilde{G}$ subject to $\tilde{P}^\top \tilde{x} = v$ satisfies $\nabla_{\tilde{x}}\tilde{G} + \tilde{P}\lambda = 0$, which recovers exactly the AL reference state equations \eqref{eq:AL_KKT_ref}. More generally, the KKT systems and learning update for AL have the same algebraic structure as EP applied to $\tilde{G}$: the bordered Hessian takes the same form, and $\partial_{k_i}\tilde{G} = \partial_{k_i}G$ since the augmented term does not depend on $k$. Setting $\tilde{P}^\top := [P^\top \ 0_{O \times O}]$, $\tilde{Q}^\top := [0_{O \times N} \ I_{O \times O}]$, and $\tilde{w}:=0$, all estimates which apply to EP in the sequel apply equally to AL with the appropriate substitutions.

\paragraph{Summary and comparison}
Table~\ref{tab:comparison} summarizes the three methods. A pattern emerges: gradient flow is recovered precisely when the error and the nudge are of dual type (voltage/current or current/voltage). When both are of the same type, as in CL, the dynamics deviate from gradient flow. This suggests a general principle: to obtain gradient descent from a contrastive learning rule, the nudge should cross the voltage--current duality relative to the error. For instance, a method with current inputs, current outputs, and current error would need to nudge with a voltage perturbation to recover gradient flow.
\begin{table}[h]
\centering
\begin{tabular}{|l|c|c|c|c|}
\hline
 & EP & CL & AL & ? \\
\hline
Input & Voltage & Voltage & Voltage & Current \\
Output & Voltage & Voltage & Voltage & Current \\
\rowcolor{gray!20} Error & Voltage & Voltage & Current & Current \\
\rowcolor{gray!20} Nudge & Current & Voltage & Voltage & Voltage \\
\rowcolor{gray!20} Gradient flow? & Yes & No & Yes & Yes \\
\hline
\end{tabular}
\caption{Comparison of EP, CL, and AL, along with a predicted fourth method (?). Gradient flow is recovered when the error and nudge cross the voltage--current duality.}
\label{tab:comparison}
\end{table}

\section{Coercivity Estimates}
\label{sec:coercivity}
As the number of trainable parameters, $M$, generally outnumbers the dimension of the desired outputs, $O$, there is, generically, a solution manifold of dimension $M-O$. We define the solution manifold as 
\begin{equation}
    \mathcal{S} := \{ k \in \mathbb{R}^M \ |  \ Q^\top x_0 =w\}
\end{equation}
where $x_0$ is the free state. Equivalently 
\begin{equation}
    \mathcal{S} := \{ k \in \mathbb{R}^M \ |  \ \mu_0 = 0\}
\end{equation}
where $\mu_0$ is the adjoint error of the reference state. \textit{We assume it is non-empty}. 

The aim of this section is to check when the dynamics defined by EP, CL, and AL are coercive: they give rise to exponentially fast decay of the loss near regular points of the solution manifold. Answering this question requires us to specify the nature of the energy function $G$. Therefore we answer these questions in the context of linear circuits with trainable conductances. Since our results are local, we expect that linearity may be dropped when $G$ is sufficiently smooth. The circuit picture also allows us to frame our results in terms of simple features of the physical state of the circuit and give physical interpretations to our results. 

\paragraph{Circuit Set-up}
To accommodate the network structure of a circuit, we introduce graph notation. Let $\mathcal{V}=\{1,...,N\}$ be the set of vertices or nodes. The state vector, $x=(x^1,x^2,\ldots,x^N) \in \mathbb{R}^N$, gives voltages at nodes. The edges are given by $\mathcal{E}=\{\{i,j\} \ | \ i \sim j\},$ and we still assume $|\mathcal{E}| =M$. The conductances $\{k_{i,j}\}_{\{i,j\} \in \mathcal{E}}$ are assumed to be the training parameters. There is one for each edge and we assume symmetry: $k_{i,j}=k_{j,i}$.

Let $\{e_i\}_{i \in \mathcal{V}}$ be the natural basis. We define the incidence vectors as $\{d_{i,j}\}_{\{i,j\} \in \mathcal{E}}$ where 
\[d_{i,j} := \begin{cases} e_i-e_j & i<j
\\ 0 & \text{else}.
\end{cases}\]
We must arbitrarily pick the incidence direction, and so we have opted for $i<j$. Then for $e \in \mathcal{E}$ s.t. $e =\{i,j\}$, we let $d_e:=d_{i,j}$. Analogously $k_e =k_{i,j}.$ We may lexicographically order the set of edges and use this ordering to define the $M$-tuple of conductances $k$. Then $e$ can be construed as the edge or the index of this edge where $e \in \{1,2,\ldots,M\}.$

We define the \textit{circuit Laplacian} as
\begin{equation}
   L(k)=\sum_{e \in \mathcal{E}}k_ed_{e}d_{e}^\top.
\end{equation}
Assume $k>0$, i.e.\ all components of $k$ are positive, and that the circuit is connected. Then $L$ is symmetric and positive semi-definite with $\ker L = \operatorname{span}\{\mathbf{1}\}$.

The energy function is
\begin{equation}
    G(x,k) := \tfrac{1}{2}\langle x, L(k)x\rangle,
\end{equation}
which is convex in $x$. We assume that the input matrix $P$ is chosen so that the constrained minimization of $G$ subject to $P^\top x = v$ has a unique solution; equivalently, $\ker L \cap \ker P^\top = \{0\}$. To handle the input constraints, we define the \textit{input-constrained Laplacian} as
\begin{equation}
    \hat{L}(k) = \begin{bmatrix}
        L & P \\ P^\top & 0
    \end{bmatrix}.
\end{equation}
Since $L$ is positive definite on $\ker P^\top$, $\hat{L}$ is invertible. Setting $\hat{x} :=(x,\lambda)$ and $\hat{v}:=(0_{N},v)$, where $\lambda$ are the Lagrange multipliers for the input constraints, Kirchhoff's law gives $\hat{L}\hat{x}_0 = \hat{v}$. The multiplier $P\lambda$ may be interpreted as the current injection needed to sustain the input nodes at the constrained voltages $v$.

The error in this context is $r= Q^\top x_0-w$ with $w$ the desired output. Below, it is helpful to have the notation $\hat{d}^\top :=[d^\top\ 0_{I}]$ and $\hat{Q}^\top := [Q^\top \ 0_{O\times I}]$.

\paragraph{Coercivity of EP and Non-Degeneracy of the Loss}
A regular point of the solution manifold is one where $\nabla_k r$ has full row rank. At such points, the Implicit Function Theorem guarantees that $\mathcal{S}$ is a smooth manifold of dimension $M-O$. As EP does gradient flow of $\Phi$, showing coercivity also gives non-degeneracy. 

First we compute 
\begin{equation}
    \partial_{k_e}r= -\hat{Q}^\top \hat{L}^{-1}\left(\partial_{k_e}\hat{L} \right)\hat{L}^{-1}\hat{v} \in \mathbb{R}^O
\end{equation}
Then $\nabla_k r \in \mathbb{R}^{O \times M}$ and is full row rank exactly when $(\nabla_k r)(\nabla_k r)^\top$ is positive definite. Taking a test vector $\tilde{r} \in \mathbb{R}^{O} $, we see that 
\begin{equation}
\begin{aligned}
  \left\|(\nabla_k r)^\top \tilde{r}\right\|_2^2&=  \sum_{e \in \mathcal{E}}\left(\tilde{r}^\top \hat{Q}^\top \hat{L}^{-1}\left(\partial_{k_e}\hat{L} \right)\hat{L}^{-1}\hat{v} \right)^2
  \\
  &= \sum_{e \in \mathcal{E}}\left(\tilde{r}^\top \hat{Q}^\top \hat{L}^{-1}\left(\hat{d}_e\hat{d}_e^\top \right)\hat{L}^{-1}\hat{v} \right)^2
  \\&= \sum_{e \in \mathcal{E}}\left(\hat{d}^\top _e  \hat{L}^{-1}\hat{Q}\tilde{r}\right)^2\left(\hat{d}_e^\top \hat{L}^{-1}\hat{v} \right)^2
  \\&= \sum_{e \in \mathcal{E}}\left(\Delta_ey(\tilde{r})\right)^2\left(\Delta_ex_0 \right)^2
   \end{aligned}
\end{equation}
Here $\Delta_ex_0$ is the voltage drop across edge $e$ in the free state. $\Delta_ey(\tilde{r})$ is the voltage drop across edge $e$, when $\tilde{r}$ is inputted at output nodes via $Q$ as a current and input nodes are grounded. In what follows, this is what we call the \textit{clamped state}. The expression on the right-hand side is exactly the update for EP found in \eqref{eq:dynamics-EP} upon replacing $\tilde{r}$ by $r$. The loss is strictly decreasing whenever this quantity is strictly positive. In the sequel, we see that this quantity can be $0$ even when a seemingly-easy objective has not been reached. 

Now we work toward a necessary and sufficient condition to establish coercivity. Then we give some simple, but only sufficient physical interpretation. First, we define a set of \textit{active edges} and associated diagonal matrix:
\begin{equation}
    \mathcal{E}' = \{e \in \mathcal{E} \ | \  \Delta_ex_0 \neq 0 \}, \quad \Lambda:=\text{diag}(\{(\Delta_ex_0)^2\}_{e \in \mathcal{E}'})
\end{equation}
We order the diagonal matrix with respect to the ordering on $\mathcal{E}'$ inherited from the lexicographical ordering on $\mathcal{E}.$ Then we define the\textit{ active-incidence matrix} as 
\[\mathscr{D}' := [\hat{d}_e]_{e \in \mathcal{E}'}.\] 

Next, we observe that 
\begin{equation}
\begin{aligned}
    \left\|(\nabla_k r)^\top \tilde{r}\right\|_2^2 &= \|\Lambda^{1/2} \mathscr{D}'^\top \hat{L}^{-1}\hat{Q}\tilde{r}\|_2^2
    \\ 
    &\ge \lambda_{\min}\left(\Lambda\right)\sigma_{\min}\left(\mathscr{D}'^\top \hat{L}^{-1}\hat{Q}\right)^2\|\tilde{r}\|_2^2,
    \end{aligned}
\end{equation}
where $\lambda_{\min} (\cdot)$ denotes the smallest eigenvalue of a square matrix and $\sigma_{\min} (\cdot)$ its smallest singular value. 
Thus coercivity follows if $\mathscr{D}'^\top\hat{L}^{-1}\hat{Q}$ has full column rank. On the other hand, if this matrix does not have full column rank, there is a test vector, $\tilde{r}$, that makes the right hand side $0$. This gives the following lemma:
\begin{lemma}
\label{lem:coercivity_iff}
    Suppose $k_0>0$. Then there exists a positive constant $c$ such that
    \[ \left\|(\nabla_k r)^\top \tilde{r}\right\|_2^2\ge c\,\|\tilde{r}\|_2^2\]
    holds for all $\tilde{r} \in \mathbb{R}^O$ and all $k$ in a neighborhood of $k_0$ if and only if the active-incidence matrix $\mathscr{D}'$ at $k_0$ is such that $\mathscr{D}'^\top\hat{L}^{-1}\hat{Q}$ has full column rank. In particular, the coercivity condition is equivalent to $\nabla_k r(k_0)$ having full row rank.
\end{lemma}

\begin{theorem}[Coercivity of EP]
\label{thm:EP_coercivity}
    Suppose $k_0>0$ and that $\mathscr{D}'^\top\hat{L}^{-1}\hat{Q}$ has full column rank at $k_0$. Set $\lambda := \lambda_{\min}(\Lambda(k_0))$ and $c := \sigma_{\min}\bigl(\mathscr{D}'^\top\hat{L}^{-1}\hat{Q}\bigr)\big|_{k_0}$. Then there is a neighborhood $U$ of $k_0$ such that for all $k \in U$ the loss $\Phi = \frac{1}{2}\|r\|^2$ satisfies
\begin{equation}
    \dot{\Phi} \leq -2\lambda c^2\,\Phi.
\end{equation}
In particular, $\Phi$ decays exponentially as long as $k$ remains in $U$.
\end{theorem}

\begin{remark}
A simple sufficient condition for coercivity is that every edge carries a nonzero voltage drop in the free state, i.e.\ $\mathcal{E}' = \mathcal{E}$. In this case $\mathscr{D}' = \mathscr{D}$ is the full incidence matrix of the graph, which has rank $N$ for a connected graph, and the coercivity condition is automatically satisfied.
\end{remark}

\paragraph{Coercivity of CL and a Local Lyapunov Function}
Considering just the first term appearing on the right-hand side of \eqref{eq:PhiD_dot}, in the circuit setting this is
\begin{equation}
\begin{aligned}
\sum_{i=1}^M\big\langle R Q D\tilde{r},\partial_{k_i}\nabla_x G\big\rangle^2
& =
   \sum_{e\in \mathcal{E}}\langle  \tilde{r}, DQ^\top (S^\top \hat{L}^{-1}S)d_ed_e^\top x\rangle^2\\
   & = \sum_{e\in \mathcal{E}}\langle  \tilde{r}, D(Q^\top S^\top)  \hat{L}^{-1}(Sd_e)(d_e^\top S^\top) \hat{L}^{-1}\hat{v}\rangle^2\\
   & = \sum_{e \in \mathcal{E}}\left(\tilde{r}^\top D \hat{Q}^\top \hat{L}^{-1}\left(\partial_{k_e}\hat{L} \right)\hat{L}^{-1}\hat{v} \right)^2
\\
& =\sum_{e \in \mathcal{E}'}\left(\hat{d}_e^\top\hat{L}^{-1}\hat{v} \right)^2\left(\tilde{r}^\top D \hat{Q}^\top \hat{L}^{-1}\hat{d}_e^\top \right)^2
\\
&\ge \lambda_{\min}\left(\Lambda\right)\sigma_{\min}\left(\mathscr{D}'^\top \hat{L}^{-1}\hat{Q}D\right)^2\|\tilde{r}\|_2^2.
\end{aligned}
\end{equation}

Up to multiplication of $\tilde{r}$ by $D = (\hat{Q}^\top\hat{L}^{-1}\hat{Q})^{-1}$, the right-hand side of the first line is exactly in the form of the expression for EP. Thus CL and EP are coercive under the same condition on the active-incidence matrix.

\begin{lemma}
\label{lem:CL_coercivity_iff}
Suppose $k_0>0$. Then there exists a positive constant $c_D$ such that
\[
\sum_{i=1}^M\big\langle R Q D\tilde{r},\partial_{k_i}\nabla_x G\big\rangle^2 \geq c_D\,\|\tilde{r}\|_2^2
\]
holds for all $\tilde{r} \in \mathbb{R}^O$ and all $k$ in a neighborhood of $k_0$ if and only if $\mathscr{D}'^\top\hat{L}^{-1}\hat{Q}$ has full column rank at $k_0$.
\end{lemma}

The other term in \eqref{eq:PhiD_dot} needs to be bounded. In the circuit setting, we have that
\begin{equation*}
  \langle r, \dot{D}r\rangle = \sum_{e \in \mathcal{E}}\left(r^\top \partial_{k_e}Dr\right)\left(\hat{d}_e^\top \hat{L}^{-1}\hat{Q}Dr\right)\left(\hat{d}_e^\top \hat{L}^{-1}\hat{v}\right).
\end{equation*}
We compute
\[
\begin{aligned}\partial_{k_e}D &= -(\hat{Q}^\top \hat{L}^{-1}\hat{Q})^{-1} \hat{Q}^\top \partial_{k_e}\hat{L}^{-1}\hat{Q}(\hat{Q}^\top\hat{L}^{-1}\hat{Q})^{-1}
\\
&=(\hat{Q}^\top \hat{L}^{-1}\hat{Q})^{-1} \hat{Q}^\top \hat{L}^{-1}\left(\hat{d}_e\hat{d}_e^\top\right)\hat{L}^{-1}\hat{Q}(\hat{Q}^\top\hat{L}^{-1}\hat{Q})^{-1}
\\
&=D\hat{Q}^\top \hat{L}^{-1}\left(\hat{d}_e\hat{d}_e^\top\right)\hat{L}^{-1}\hat{Q}D.
\end{aligned}
\]
Hence
\begin{equation}
\begin{aligned}
 \left| \langle r, \dot{D}r\rangle \right|&= \left|\sum_{e \in \mathcal{E}}\left(\hat{d}_e^\top \hat{L}^{-1}\hat{Q}Dr\right)^3\left(\hat{d}_e^\top \hat{L}^{-1}\hat{v}\right)\right|
  \\& \leq 4|\mathcal{E}|\sigma_{\max}(\hat{L}^{-1}\hat{Q}D)^3\sigma_{\max}(\hat{L}^{-1})\|v\|\|r\|^3.
\end{aligned}
\end{equation}

Combining the coercivity estimate with the cubic bound on $\langle r, \dot{D}r\rangle$ gives the following.

\begin{theorem}[Coercivity of CL]
\label{thm:CL_coercivity}
    Suppose $k_0>0$ and that $\mathscr{D}'^\top\hat{L}^{-1}\hat{Q}$ has full column rank at $k_0$. Let $\lambda_D, c_D > 0$ denote the coercivity constants for $\Phi_D$ at $k_0$, i.e., constants such that $\|\nabla\Phi_D\|^2 \ge 2\lambda_D c_D^2\,\Phi_D$ on a neighborhood of $k_0$ (these are assembled from Lemma~\ref{lem:CL_coercivity_iff} and the bounds on $\lambda_{\min}(D), \lambda_{\max}(D)$). Then there is a neighborhood $U$ of $k_0$ and a constant $C_D > 0$ such that for all $k \in U$ the weighted loss $\Phi_D = \frac{1}{2}\langle r, Dr\rangle$ satisfies
\begin{equation}
    \dot{\Phi}_D  \leq -2\lambda_D c_D^2\,\Phi_D + C_D\,\Phi_D^{3/2}.
\end{equation}
In particular, if $k_0 \in \mathcal{S}$ (so that $\Phi_D(k_0) = 0$), then by shrinking $U$ if necessary,
\begin{equation}
    \dot{\Phi}_D \leq -\lambda_D c_D^2\,\Phi_D
\end{equation}
for all $k \in U$.
\end{theorem}

\paragraph{Physical Interpretation for AL}
The learning update for AL has the same product-of-voltage-drops structure as EP, so the coercivity argument applies with the active edge set defined according to the fully-constrained reference state rather than the free state. Here we give the physical interpretation. In AL, the reference state $\tilde{x}_0 = (x_0,\mu_0)$ solves a Poisson equation:
\begin{equation*}
    \begin{bmatrix}
        L & P  & Q
        \\ P^\top & 0&  0
        \\ Q^\top & 0 &0
    \end{bmatrix}
\begin{bmatrix}
    x_0 \\ \lambda_0 \\ \mu_0
\end{bmatrix}
=
 \begin{bmatrix}
        0 \\ v \\ w
         \end{bmatrix}.
\end{equation*}
Therefore, in the reference state, the inputs are clamped in the usual way but the outputs are also clamped to the target $w$.

The clamped state in the small nudge limit is given by
\begin{equation*}
    \begin{bmatrix}
        L & P  & Q
        \\ P^\top & 0&  0
        \\ Q^\top & 0 &0
    \end{bmatrix}
\begin{bmatrix}
    y\\ \alpha \\ \beta
\end{bmatrix}
=
 \begin{bmatrix}
        0 \\ 0 \\ \mu_0
         \end{bmatrix}.
\end{equation*}
Therefore the input is grounded and the output is clamped to the \textit{current error} $\mu_0$. The ODE describing training for edge $e$ is given by
\begin{equation}
    \dot{k}_e = \Delta_e y(\mu_0) \Delta_ex_0 =-\frac{1}{2} \partial_{k_e} \|\mu_0\|^2
\end{equation}
and so AL performs gradient flow on the \textit{current error loss}. Since the active edge set is now defined according to the reference state $x_0$, the quantities $\mathscr{D}'$ and $\Lambda$ are defined in terms of this new active edge set. The coercivity constant becomes $\lambda_{\min}(\Lambda)\sigma_{\min}\left(\mathscr{D}'^\top \hat{L}^{-1}\hat{Q}D\right)^2$, formally the same expression as for CL. This follows after solving the clamped system: $y = \hat{L}^{-1}\hat{Q}D\mu_0$.

\section{Loss of Coercivity and Genericity}
\label{sec:kite}

The coercivity condition of the previous section guarantees exponential decay of the loss, but it is natural to ask whether it always holds. In this section we show that it can fail: we give a physical interpretation of the coercivity condition for single-output circuits, exhibit a concrete circuit (the ``kite'') where coercivity degenerates, and then prove that such failures are non-generic. We close with numerical experiments illustrating the coercivity landscape and its effect on convergence rates.

\paragraph{Physical Interpretation for Loss of Coercivity when $O=1$}
When $O=1$, the matrix $\mathscr{D}'^\top \hat{L}^{-1}\hat{Q}$ is a column vector. Its $e$-th entry is $\Delta_e \hat{y}$, the voltage drop across active edge $e$ in the clamped response to a unit output current $\hat{Q}$. The coercivity condition (full column rank) therefore reduces to the requirement that at least one active edge carries a nonzero clamped voltage drop. In other words, coercivity fails if and only if the free and clamped voltage drops have disjoint support: every edge that carries current in the free state is invisible to the clamped state, and vice versa.

This gives an explicit factorization of the loss decay. Since the residual is the scalar $r = Q^\top x_0 - w$, the EP update takes the form $\dot{k}_e = (\Delta_e x_0)(\Delta_e \hat{y})\, r$, and the loss $\Phi = \frac{1}{2}r^2$ satisfies
\[
\dot{\Phi} = -r^2 \sum_{e \in \mathcal{E}'} (\Delta_e x_0)^2\, (\Delta_e \hat{y})^2 = -2\Phi\, c(k),
\]
where the \emph{coercivity constant} $c(k) := \sum_{e \in \mathcal{E}'} (\Delta_e x_0)^2\, (\Delta_e \hat{y})^2$ depends only on the conductances, not on the residual. Coercivity ($c(k) > 0$) is therefore equivalent to local exponential decay of the loss.

For CL, the Dirichlet-to-Neumann map $D$ is a positive scalar when $O=1$, so $\mathscr{D}'^\top \hat{L}^{-1}\hat{Q}D = D\,\mathscr{D}'^\top \hat{L}^{-1}\hat{Q}$ and the rank condition is identical to that of EP. Thus, for single-output circuits (but also in general), EP and CL lose coercivity under exactly the same conditions.

\paragraph{Example of Degeneracy}

We demonstrate coercivity failure with the kite circuit in Figure~\ref{fig:Circuit_with_failure}. Let node $1$ be the input with $x_1 = 1$ and node $5$ be grounded. Let the output be $Q^\top x = x_2 - 2x_3 + x_4$ with target $w$.

Suppose the conductances satisfy the ratio condition $k_{1,2}/k_{2,5} = k_{1,4}/k_{4,5}$. Then nodes $2$ and $4$ see the same voltage divider ratio between the source and ground, forcing $x_2 = x_4$. Since edges $\{2,3\}$ and $\{3,4\}$ then carry no net current, we also get $x_3 = x_2$. Therefore $Q^\top x_0 = 0$ and the only active edges (nonzero free-state voltage drops) are $\{1,2\}$, $\{1,4\}$, $\{2,5\}$, and $\{4,5\}$.

The condition $k_{2,3} = k_{3,4}$, which is independent of the ratio condition, also causes $Q^\top x = 0$. This is because the middle branch acts as a voltage divider and equality of the two conductances means the voltage divides exactly in half. The intersection of this condition with the ratio condition is where degeneracy occurs.  

\paragraph{Loss of coercivity on the solution manifold ($w = 0$)} The system is on the solution manifold $\{Q^\top x_0 = 0\}$, since the ratio condition alone (codimension $1$ in $k$-space) forces $Q^\top x_0 = 0$. At generic points of this ratio set, coercivity holds. For coercivity to fail, the clamped voltage drops must also vanish on all active edges, i.e.\ $y_2 = y_4 = 0$. One checks that this only occurs when $k_{2,3} = k_{3,4}$. The non-coercive points on the solution manifold are therefore those satisfying both the ratio condition and $k_{2,3} = k_{3,4}$: a codimension-$2$ subset of $k$-space. At these points, $r = 0$ and the loss vanishes, but the learning dynamics are slow near these points, as the exponential convergence guarantee is lost.

\paragraph{Spurious fixed points ($w \neq 0$).} Now $r = -w \neq 0$. The same conductances (ratio condition plus $k_{2,3} = k_{3,4}$) produce the same disjoint support between free and clamped voltage drops. The update rule vanishes identically: the system is stuck at nonzero loss. The multiplier $D$ cannot resolve this for CL, since it only rescales the forcing without changing its support. These are spurious fixed points of the learning dynamics that lie off the solution manifold.

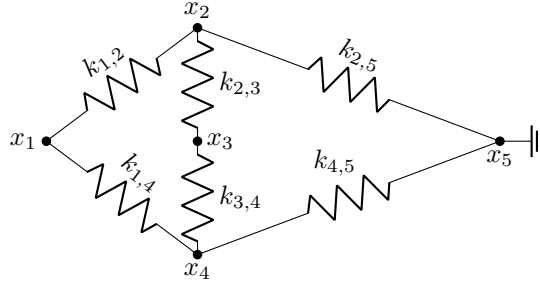
\begin{figure}[ht]
  \centering
  \begin{circuitikz}
  \coordinate (x1) at (0,0);      
  \coordinate (x2) at (2,1.5);    
  \coordinate (x4) at (2,-1.5);   
  \coordinate (x3) at (2,0);      
  \coordinate (G)  at (6,0);      

\node[circ] (x1) at (0,0) {};
\node[circ] (x2) at (2,1.5) {};
\node[circ] (x3) at (2,0) {};
\node[circ] (x4) at (2,-1.5) {};
\node[circ] (x5) at (6,0) {};
\node        (G)  at (6,0) {};
  \node[left]  at (x1) {$x_1$};
  \node[above] at (x2) {$x_2$};
  \node[right] at (x3) {$x_3$};
  \node[below] at (x4) {$x_4$};
   \node[below] at (x5) {$x_5$};

  \draw (x1) to[R=$k_{1,2}$] (x2);
  \draw (x1) to[R=$k_{1,4}$] (x4);
  \draw (x2) to[R=$k_{2,3}$] (x3);
  \draw (x3) to[R=$k_{3,4}$] (x4);

  \draw (x2) to[R=$k_{2,5}$] (x5);
  \draw (x4) to[R=$k_{4,5}$] (x5);
  \node[ground, rotate=90] at (G) {};
\end{circuitikz}
  \caption{A kite circuit. Under the ratio condition $k_{1,2}/k_{2,5} = k_{1,4}/k_{4,5}$ and $k_{2,3}=k_{3,4}$, the free and clamped voltage drops have disjoint support and coercivity fails.}
  \label{fig:Circuit_with_failure}
\end{figure}

\paragraph{Comparison with AL}
We find that AL produces similar albeit not exactly the same degeneracies. In the case of $w = 0$, we find the same loss of coercivity on the solution manifold. This is a result of the fact that on the solution manifold, the free state of EP and reference state of AL are identical and so AL provides nothing new to remove the non-coercive set. 

 When $w \ne 0$, the degenerate points which were spurious fixed points for EP are no longer fixed points for AL. Clamping the output to $w$ in the reference state forces voltage drops across $k_{2,3}$ and $k_{3,4}$, which were absent in the EP free state. The AL update therefore acts nontrivially on $k_{2,3}$ and $k_{3,4}$, while the four outer conductances remain fixed for the same reason they were fixed under EP. The update is moreover symmetric under the swap $k_{2,3} \leftrightarrow k_{3,4}$, so $k_{2,3} = k_{3,4}$ is preserved. The dynamics are therefore confined to the codimension-$2$ invariant manifold on which the ratio condition and $k_{2,3} = k_{3,4}$ hold, and on which the output never changes from $0$. However, this is not a contradiction with AL being gradient flow; it is just that it is gradient flow of the current error. The current error scales with $\|k\|$, and the degenerate dynamics serve only to decrease $\|k\|$. So $k_{2,3}$ and $k_{3,4}$ simply decrease to $0$, which is a local minimum of the current error loss at the boundary. 

\paragraph{Genericity of coercivity}  The $w=0$ case of the kite example shows that the coercivity condition can fail on the solution manifold for special choices of target. However, such failures are non-generic. Consider the map
\[
\Psi(k) := Q^\top x_0(k)
\]
from $\{k>0\}$ to $\mathbb{R}^O$. Since $x_0(k) = \hat{L}^{-1}(k)\hat{v}$, the map $\Psi$ is real-analytic (in fact, rational) in $k$. The solution manifold for target $w$ is $\mathcal{S} = \Psi^{-1}(w)$, and the coercivity condition at $k \in \mathcal{S}$ is exactly the condition that $\nabla_k\Psi(k)$ is surjective---that is, $w$ is a regular value of $\Psi$ at $k$.

\begin{proposition}[Genericity of Coercivity]
\label{prop:genericity}
Define the set of \emph{critical targets}
\[
\mathcal{W}_{\rm crit} := \Psi\big(\{k > 0 : \operatorname{rank} \nabla_k\Psi(k) < O\}\big) \subset \mathbb{R}^O.
\]
Then $\mathcal{W}_{\rm crit}$ has Lebesgue measure zero in $\mathbb{R}^O$. For any $w \notin \mathcal{W}_{\rm crit}$, the solution manifold $\mathcal{S} = \Psi^{-1}(w)$ is a smooth submanifold of dimension $M-O$ and the coercivity condition holds at every point of $\mathcal{S}$.
\end{proposition}

\begin{proof}
$\mathcal{W}_{\rm crit}$ is the set of critical values of the smooth map $\Psi$. By Sard's theorem (see \cite{bates1993} for precise smoothness hypotheses), $\mathcal{W}_{\rm crit}$ has Lebesgue measure zero in $\mathbb{R}^O$. For $w \notin \mathcal{W}_{\rm crit}$, every $k \in \Psi^{-1}(w)$ satisfies $\operatorname{rank} \nabla_k\Psi(k) = O$, which is the coercivity condition. The Implicit Function Theorem then gives that $\Psi^{-1}(w)$ is a smooth manifold of dimension $M-O$.
\end{proof}

Note that the rank condition $\operatorname{rank}\nabla_k\Psi = O$ is exactly the coercivity condition of Theorem~\ref{thm:EP_coercivity}, for which Lemma~\ref{lem:coercivity_iff} provides a concrete necessary and sufficient criterion. If coercivity holds at any single point $k \in \{k > 0\}$, then $\Psi$ is a submersion at $k$, so its image contains an open set and in particular has positive Lebesgue measure. Since $\mathcal{W}_{\rm crit}$ has measure zero, the set of achievable regular values $\operatorname{range}(\Psi) \setminus \mathcal{W}_{\rm crit}$ has positive measure. The proposition is therefore non-vacuous whenever the circuit admits at least one coercive configuration.
\paragraph{Numerical Experiment}
We illustrate the coercivity landscape and its effect on convergence with numerical experiments on the kite circuit. We parameterize conductances as
\[
(k_{1,2},\; k_{1,4},\; k_{2,3},\; k_{3,4},\; k_{2,5},\; k_{4,5}) = (a + \alpha,\; a,\; c + \tfrac{\beta}{2},\; c - \tfrac{\beta}{2},\; b,\; b)
\]
with base values $a = 2$, $b = 1.5$, $c = 2$, so that $\alpha = k_{1,2} - k_{1,4}$ measures deviation from the ratio condition and $\beta = k_{2,3} - k_{3,4}$ measures deviation from $k_{2,3} = k_{3,4}$. The non-coercive set sits at $(\alpha, \beta) = (0,0)$.

Figure~\ref{fig:coercivity_heatmap} shows the coercivity constant $c(k) = \sum_{e \in \mathcal{E}'} (\Delta_e x_0)^2 (\Delta_e \hat{y})^2$ on this two-parameter slice. The coercivity degenerates only at the origin (a codimension-$2$ point), and the level sets are roughly elliptical, elongated in the $\alpha$-direction. The circuit is more sensitive to breaking $k_{2,3} = k_{3,4}$ than to breaking the ratio condition. The white dashed contours show the level sets of the output map $\Psi(k) = Q^\top x_0$. The $w = 0$ contour passes through the non-coercive point, confirming that $0 \in \mathcal{W}_{\rm crit}$, while the $w \neq 0$ contours avoid it. This demonstrates Proposition~\ref{prop:genericity}: perturbing the target $w$ away from the critical value shifts the solution manifold clear of the non-coercive set, ensuring that every point on the manifold is coercive.

Figure~\ref{fig:convergence_semilog} shows the EP learning dynamics started from a fixed but randomly chosen perturbation of $k^*$ at several values of $\beta$ (with $\alpha = 0$, $w = 0$). On the semi-log scale, the loss decays as a straight line in every case, confirming exponential convergence. The decay rate increases with $\beta$, tracking the coercivity constant $c(k^*)$. For a fixed target $w$, one can improve the rate by moving $k$ away from the non-coercive set, but near the degenerate point convergence is inherently slow.

Shifting the target $w$ is a natural remedy: for a classification task, this amounts to choosing a different label encoding, which is easy to do in practice. By Proposition~\ref{prop:genericity}, a generic choice of $w$ makes the solution manifold fully coercive, guaranteeing local exponential convergence near every point on the manifold. However, the theory is local: the convergence ball must be small enough to avoid the spurious fixed points discussed above, which lie off the manifold but are still fixed points of the EP dynamics. From a global convergence perspective, shifting $w$ alone may not suffice if the initial condition is near such a point. On the other hand, the non-coercive set and the associated spurious fixed points are both codimension-$2$ in $k$-space, so generic trajectories should miss them. 



\begin{figure}[ht]
  \centering
  \includegraphics[width=0.75\textwidth]{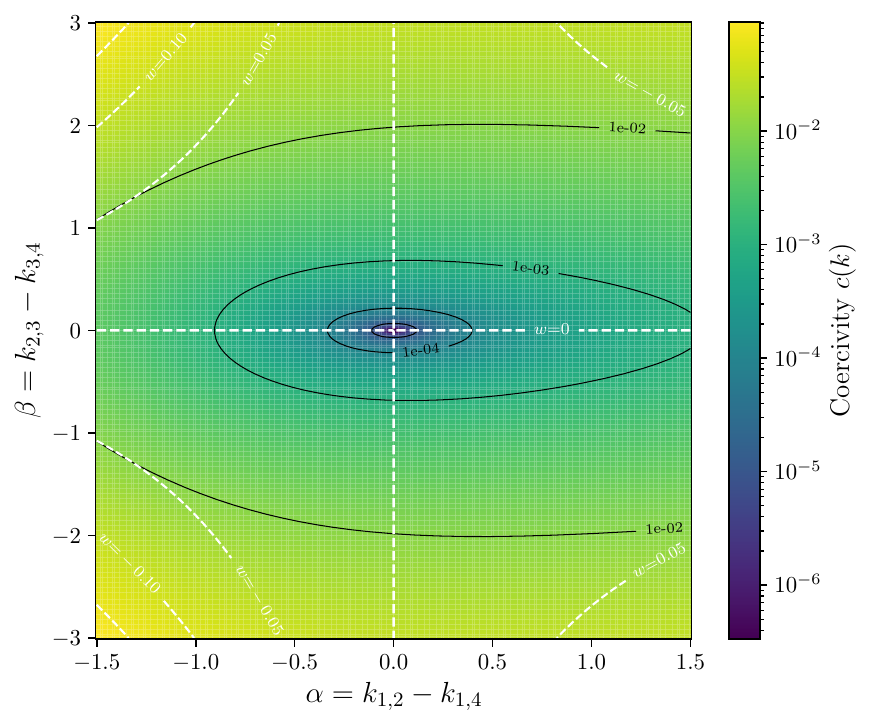}
  \caption{Coercivity constant $c(k)$ on the kite circuit, parameterized by $\alpha = k_{1,2} - k_{1,4}$ and $\beta = k_{2,3} - k_{3,4}$. White dashed curves are level sets of the output $w = Q^\top x_0$; black curves are level sets of $c(k)$.}
  \label{fig:coercivity_heatmap}
\end{figure}

  \begin{figure}[ht]                                                                                                    
    \centering                                              
    \begin{minipage}[t]{0.48\textwidth}
      \centering                                                                                                        
      \includegraphics[width=\textwidth]{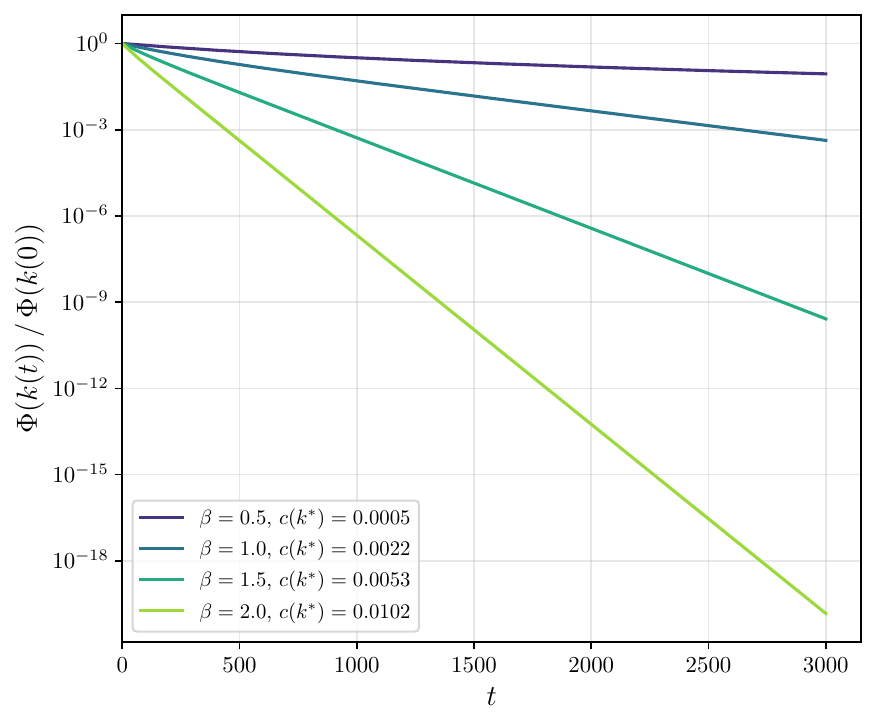}\\
      (a)                                                                                                               
    \end{minipage}\hfill                                    
    \begin{minipage}[t]{0.48\textwidth}                                                                                 
      \centering                                                                                                        
      \includegraphics[width=\textwidth]{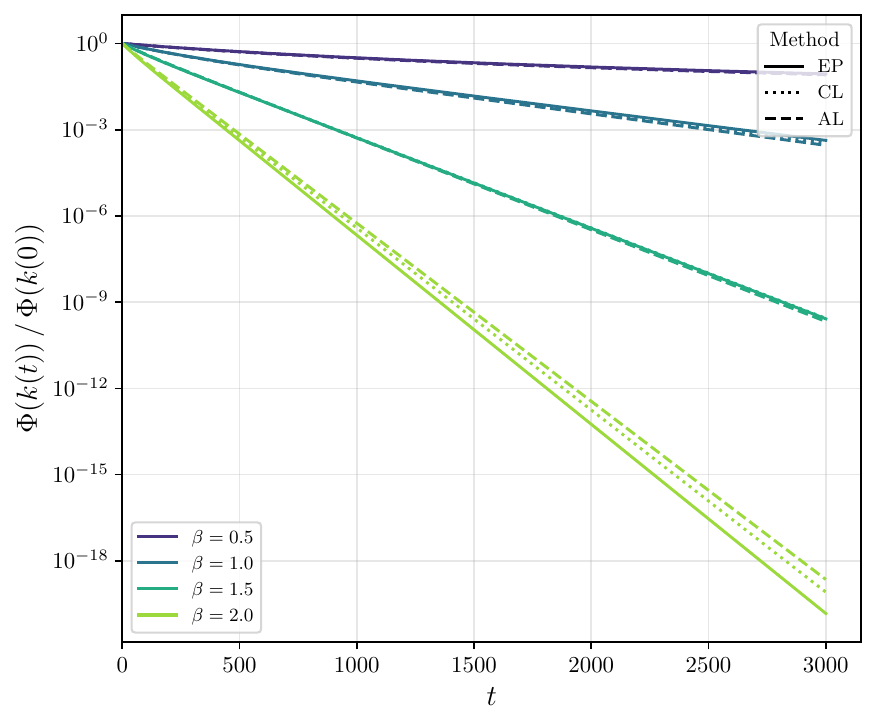}\\
      (b)                                                                                                               
    \end{minipage}                                          
    \caption{Normalized loss $\Phi(k(t))/\Phi(k(0))$ on the kite circuit, started from perturbations of $k^*$ on the    
  $w=0$ manifold at $\alpha = 0$ and varying $\beta$. (a) EP dynamics: convergence is exponential at every $\beta > 0$, 
  with rate controlled by the coercivity constant $c(k^*)$. (b) Comparison of EP, CL, and AL from the same initial      
  conditions; all three methods inherit the same coercivity-controlled exponential decay, with rates rescaled by factors
   of $D(k^*)$ and $D(k^*)^2$ respectively.}                
    \label{fig:convergence_semilog}
  \end{figure}

\section{Dynamic Stability and Local Convergence}
\label{sec:convergence}

\begin{remark}
For EP, the full row rank condition on $\nabla_k r$ rules out spurious fixed points. Since EP is gradient flow of $\Phi = \frac{1}{2}\|r\|^2$, any fixed point satisfies $\nabla_k \Phi = (\nabla_k r)^\top r = 0$. If $\nabla_k r$ has full row rank, then $(\nabla_k r)^\top$ has trivial null space, forcing $r = 0$. Thus, EP can only get stuck at a point with nonzero loss if $\nabla_k r$ is rank-deficient there, as in the kite circuit example above. The same applies to AL, which is gradient flow of the adjoint loss $\Phi^*$.
\end{remark}

\begin{theorem}\label{thm:local_convergence}
Suppose that $k^*$ is a point on the solution manifold $\mathcal{S}:=\{k \in \mathbb{R}^M \mid r(k) = 0\}$ such that $\nabla_k r$ has full row rank. (For EP and CL, this holds exactly when the active-incidence conditions of Theorems~\ref{thm:EP_coercivity} and~\ref{thm:CL_coercivity} are satisfied.) Then there exists a positive real number $\delta$ such that for all $\|k(0) - k^*\| < \delta$, we have $\lim_{t \to \infty} k(t) \in \mathcal{S}$ and $\|r(k(t))\|$ decays exponentially.
\end{theorem}

\begin{proof}
Since $\nabla_k r(k^*)$ has full row rank, we may reorder coordinates and write
\[
k = (u,v), \qquad k^* = (u^*, v^*),
\]
with $u \in \mathbb{R}^{M-O}$ and $v \in \mathbb{R}^{O}$, such that
\[
E := \frac{\partial r}{\partial v}(k^*)
\]
is invertible.

By the Implicit Function Theorem, there exist open neighborhoods $U \subset \mathbb{R}^{M-O}$ of $u^*$ and $V \subset \mathbb{R}^{O}$ of $v^*$, and a $C^1$ map
\[
\psi : U \to V
\]
such that
\[
\mathcal{S} \cap (U \times V) = \{ (u, \psi(u)) : u \in U \}.
\]
In particular, $U \times V$ is an open neighborhood of $k^*$.

By shrinking $U \times V$ if necessary, we may assume the following hold uniformly for all $k \in U \times V$:
\begin{enumerate}
    \item[(i)] $\dfrac{\partial r}{\partial v}(k)$ is invertible with $\left\|\left(\dfrac{\partial r}{\partial v}(k)\right)^{-1}\right\| \le C_1$,
    \item[(ii)] $|\dot{k}| \le C_2 \|r(k)\|$,
    \item[(iii)] For EP: $\dot{\Phi} \leq -2\lambda c^2\,\Phi$ (Theorem~\ref{thm:EP_coercivity}). For CL: $\dot{\Phi}_D \leq -\lambda_D c_D^2\,\Phi_D$ (Theorem~\ref{thm:CL_coercivity}), with $D$ having uniformly bounded eigenvalues so that $\tfrac{1}{2}\lambda_{\min}(D)\|r\|^2 \leq \Phi_D \leq \tfrac{1}{2}\lambda_{\max}(D)\|r\|^2$,
    \item[(iv)] $\psi(U) \subset V$.
\end{enumerate}

Note that for $k = (u,v) \in U \times V$,
\begin{equation}
\operatorname{dist}(k,\mathcal S)
\le \|(u,v)-(u,\psi(u))\|
= \|v-\psi(u)\|.
\label{eq:dist-bound-v}
\end{equation}

Using the Fundamental Theorem of Calculus along the segment from $(u,\psi(u))$ to $(u,v)$,
\[
r(u,v)=r(u,v) - r(u,\psi(u))
= 
\int_0^1
\frac{\partial r}{\partial v}\big(u, \psi(u) + t(v-\psi(u))\big)\, dt \,(v-\psi(u)).
\]
Define
\[
A :=
\int_0^1
\frac{\partial r}{\partial v}\big(u, \psi(u) + t(v-\psi(u))\big)\, dt.
\]
Then
\begin{equation}
r(u,v) = A (v-\psi(u)).
\label{eq:r-factored}
\end{equation}
By (i), $A$ is a small perturbation of $E$, and hence invertible with $\|A^{-1}\| \le 2C_1$. From \eqref{eq:r-factored},
\[
\|v-\psi(u)\| \le \|A^{-1}\| \|r(u,v)\| \le 2C_1 \|r(k)\|.
\]

Combining with \eqref{eq:dist-bound-v},
\begin{equation}
\operatorname{dist}(k,\mathcal S) \le 2C_1 \|r(k)\|
\quad \text{for all } k \in U \times V.
\label{eq:dist-bound-r}
\end{equation}

As long as $k(t) \in U \times V$, conditions (i)--(iv) hold. By (iii), coercivity gives
\[
\dfrac{d}{dt}\|r(k(t))\|^2 \le -2\lambda c^2\, \|r(k(t))\|^2,
\]
and hence
\[
\|r(k(t))\| \le e^{-\lambda c^2\, t}\, \|r(k(0))\|
\]
for EP. For CL, the analogous inequality has
\[\|r(k(t))\| \le \sqrt{\tfrac{\lambda_{\max}(D)}{\lambda_{\min}(D)}}\,e^{-\lambda_D c_D^2\, t/2}\, \|r(k(0))\|.\]
In either case, there exist $K\ge 1$ and $\nu>0$ such that \[\|r(k(t))\| \le Ke^{-\nu t}\|r(k(0))\|\] (with $K=1,\,\nu=\lambda c^2$ for EP; $K=\sqrt{\lambda_{\max}(D)/\lambda_{\min}(D)},\,\nu = \lambda_D c_D^2/2$ for CL).          

Since $r(k^*) = 0$ and $r$ is $C^1$, there exists $C_3 > 0$ such that $\|r(k)\| \leq C_3\|k - k^*\|$ on $U \times V$. By (ii), the total displacement of $k$ is controlled:
\[
\|k(t) - k(0)\| \le C_2 \int_0^t \|r(k(s))\|\, ds \le \dfrac{KC_2}{\nu}\, \|r(k(0))\| \leq \dfrac{KC_2 C_3}{\nu}\, \delta.
\]
By the triangle inequality,
\[
\|k(t) - k^*\| \leq \|k(t) - k(0)\| + \|k(0) - k^*\| < \left(\dfrac{KC_2 C_3}{\nu}+1\right)\delta.
\]
Choose $\delta$ small enough that $\left(\dfrac{KC_2 C_3}{\nu}+1\right)\delta < \operatorname{dist}(k^*, \partial(U \times V))$. By continuity, $k(t) \in U \times V$ for $t$ in some maximal interval $[0,T)$. The bounds above hold on $[0,T)$ and show $k(t)$ remains bounded away from $\partial(U \times V)$, so $T = \infty$.

It remains to show $\lim_{t \to \infty} k(t)$ exists. For $0 \le s \le t$, the same estimate as above gives
\[
\|k(t) - k(s)\| \le C_2 \int_s^t \|r(k(\tau))\|\, d\tau
\le \dfrac{KC_2}{\nu}\, e^{-\nu\, s}\, \|r(k(0))\|,
\]
which tends to $0$ as $s \to \infty$ uniformly in $t \ge s$. Hence $\{k(t)\}$ is Cauchy and $k_\infty := \lim_{t \to \infty} k(t)$ exists. Since $\|r(k(t))\|$ decays exponentially for all $t \ge 0$, \eqref{eq:dist-bound-r} gives $\operatorname{dist}(k(t),\mathcal S) \to 0$, and since $\mathcal{S}$ is closed, $k_\infty \in \mathcal{S}$.
\end{proof}

\section{Convergence of the Discrete Learning Rule}
\label{sec:discrete}

The analysis to this point has been for the continuous-time limit $\tau \to 0$. In practice the learning rule \eqref{eq:continuous-dynamics} is applied one step at a time at a finite learning rate $\tau > 0$ and finite nudge $\eta > 0$. It is natural to ask whether the discrete dynamics inherit the exponential convergence of the continuous flow. In this section we show that, under the same coercivity condition as Theorems~\ref{thm:EP_coercivity} and \ref{thm:CL_coercivity}, the discrete EP and CL rules drive the loss to zero geometrically in the number of iterations, provided the learning rate and the initial loss are sufficiently small. The argument combines the descent lemma for smooth optimization with a Polyak--\L{}ojasiewicz (PL) inequality coming directly from coercivity.

\paragraph{Discrete update identity for EP}

The EP update \eqref{eq:continuous-dynamics} at finite nudge, specialized to the resistor-network energy $G(x,k) = \tfrac{1}{2}\langle x, L(k) x\rangle$, admits an exact identity that separates gradient descent from an $O(\eta)$ physical correction. Since $\partial_{k_e} G = \tfrac{1}{2}(\Delta_e x)^2$ and the nudged state is exactly $x_\eta = x + \eta y$ with $y := \hat{L}^{-1}\hat{Q}r$ (since $\nabla_x G$ is linear in $x$, the equilibrium conditions are linear and $x_\eta$ is affine in $\eta$), we have
\[
\frac{\partial_{k_e} G(x_\eta; k) - \partial_{k_e} G(x_0; k)}{\eta} = (\Delta_e x_0)(\Delta_e y) + \tfrac{\eta}{2}\,(\Delta_e y)^2.
\]
Recalling from Section~\ref{sec:coercivity} that $\partial_{k_e}\Phi = -(\Delta_e x_0)(\Delta_e y(r))$ with $y(r) = \hat{L}^{-1}\hat{Q}r$, the discrete update for each edge $e$ takes the form
\begin{equation}\label{eq:EP_discrete}
k_e^{t+1} - k_e^t = \tau\,(\Delta_e x_0^t)(\Delta_e y^t) + \tfrac{\tau\eta}{2}\,(\Delta_e y^t)^2,
\end{equation}
where $y^t := \hat{L}^{-1}\hat{Q}r^t$. The first term is the continuous-time gradient descent step; the second is an $O(\eta)$ correction from the finite-nudge expansion. In vector form,
\[
k^{t+1} - k^t = -\tau\,\nabla\Phi^t + \tfrac{\tau}{2}\,\varepsilon^t, \qquad \varepsilon^t \in \mathbb{R}^M, \quad \varepsilon^t_e := \eta\,(\Delta_e y^t)^2.
\]

\paragraph{Convergence theorem for EP}

\begin{theorem}\label{thm:discrete_EP}
Let $k^* \in \mathcal{S}$ satisfy the coercivity condition of Theorem~\ref{thm:EP_coercivity}, and write $\lambda := \lambda_{\min}(\Lambda(k^*))$, $c^2 := \sigma_{\min}(\mathscr{D}'^\top \hat{L}^{-1}\hat{Q})^2\big|_{k^*}$. There exist a neighborhood $U \ni k^*$, a Lipschitz constant $L_\Phi > 0$ for $\nabla\Phi$ on $U$, a constant $C > 0$ with $\|\Delta y(r)\| \le C\,\|r\|$ on $U$, and a radius $\delta > 0$ such that for any learning rate, nudge, and initial condition satisfying
\[
0 < \tau \le \tfrac{1}{L_\Phi}, \qquad \|k^0 - k^*\| < \delta, \qquad \Phi(k^0) \le \tfrac{\lambda c^2}{8\eta^2 C^4},
\]
the iterates of \eqref{eq:EP_discrete} remain in $U$ for all $t \ge 0$ and satisfy
\[
\Phi(k^{t+1}) \le \gamma\,\Phi(k^t), \qquad \gamma := 1 - \tfrac{\lambda c^2\,\tau}{4} \in (0,1).
\]
\end{theorem}

\begin{proof}
Shrink $U$ so that $\nabla \Phi$ is $L_\Phi$-Lipschitz on $U$, the norm bound $\|\Delta y(r)\| \le C\|r\|$ holds uniformly, and the coercivity estimate of Theorem~\ref{thm:EP_coercivity} gives the PL inequality
\begin{equation}\label{eq:PL}
\|\nabla \Phi(k)\|^2 \ge 2\lambda c^2\,\Phi(k).
\end{equation}
(Equation \eqref{eq:PL} is Lemma~\ref{lem:coercivity_iff} applied with $\tilde r = r$, combined with $\Phi = \tfrac{1}{2}\|r\|^2$ and the explicit constant $c = \sigma_{\min}(\mathscr{D}'^\top\hat{L}^{-1}\hat{Q})$ coming from the decomposition at the end of Theorem~\ref{thm:EP_coercivity}'s proof.) Abbreviate $\Phi^t := \Phi(k^t)$ and $\delta k := k^{t+1} - k^t = -\tau \nabla \Phi^t + \tfrac{\tau}{2}\varepsilon^t$. The descent lemma gives
\[
\Phi^{t+1} - \Phi^t \le \langle \nabla\Phi^t, \delta k\rangle + \tfrac{L_\Phi}{2}\|\delta k\|^2.
\]
Substituting the update and collecting terms,
\[
\Phi^{t+1} - \Phi^t \le \bigl(-\tau + \tfrac{L_\Phi\tau^2}{2}\bigr)\|\nabla\Phi^t\|^2 + \tfrac{\tau(1 - L_\Phi\tau)}{2}\,\langle \nabla\Phi^t, \varepsilon^t\rangle + \tfrac{L_\Phi\tau^2}{8}\|\varepsilon^t\|^2.
\]
For $\tau \le 1/L_\Phi$, the coefficient on $\|\nabla\Phi^t\|^2$ is at most $-\tau/2$, and Young's inequality applied to the cross term gives $\tfrac{\tau(1 - L_\Phi\tau)}{2}\,|\langle \nabla\Phi^t, \varepsilon^t\rangle| \le \tfrac{\tau}{4}(\|\nabla\Phi^t\|^2 + \|\varepsilon^t\|^2)$. Combining,
\begin{equation}\label{eq:desc_pl_combined}
\Phi^{t+1} \le \Phi^t - \tfrac{\tau}{4}\|\nabla\Phi^t\|^2 + \tfrac{\tau}{2}\|\varepsilon^t\|^2.
\end{equation}
The error term is controlled by $\Phi^t$: using $\sum_e (\Delta_e y^t)^4 \le \bigl(\sum_e (\Delta_e y^t)^2\bigr)^2 = \|\Delta y(r^t)\|^4$ and the norm bound $\|\Delta y(r^t)\| \le C\|r^t\|$,
\[
\|\varepsilon^t\|^2 = \eta^2\,\sum_e (\Delta_e y^t)^4 \le \eta^2\,C^4\|r^t\|^4 = 4\eta^2 C^4\,(\Phi^t)^2.
\]
Substituting this and \eqref{eq:PL} into \eqref{eq:desc_pl_combined},
\[
\Phi^{t+1} \le \bigl(1 - \tfrac{\lambda c^2\tau}{2}\bigr)\,\Phi^t + 2\tau\eta^2 C^4\,(\Phi^t)^2.
\]
Provided $\Phi^t \le \tfrac{\lambda c^2}{8\eta^2 C^4}$, the quadratic correction is at most $\tfrac{\lambda c^2\tau}{4}\Phi^t$, so
\[
\Phi^{t+1} \le \bigl(1 - \tfrac{\lambda c^2\tau}{4}\bigr)\Phi^t = \gamma\,\Phi^t.
\]
Since $\gamma < 1$, $\Phi^t \le \gamma^t\Phi^0 \le \Phi^0$ for all $t$ and the smallness condition on $\Phi^t$ propagates.

It remains to bound the trajectory. From \eqref{eq:EP_discrete} and the PL inequality $\|\nabla\Phi^t\| \le \|\Lambda\|^{1/2}\,C\,\|r^t\| =: C_1\|r^t\|$ (with $C_1$ uniform on $U$),
\[
\|k^{t+1} - k^t\| \le \tau C_1\|r^t\| + \tfrac{\tau\eta C^2}{2}\|r^t\|^2 \le \bigl(\tau C_1\sqrt{2\Phi^0} + \tau\eta C^2\,\Phi^0\bigr)\gamma^{t/2}.
\]
Summing,
\[
\|k^t - k^0\| \le \sum_{i=0}^{t-1}\|k^{i+1} - k^i\| \le \frac{\tau C_1\sqrt{2\Phi^0} + \tau\eta C^2\,\Phi^0}{1 - \sqrt{\gamma}},
\]
uniformly in $t$. Taking $\Phi^0$ (equivalently, $\delta$) small enough makes this bound smaller than $\mathrm{dist}(k^*, \partial U) - \delta$, so $k^t \in U$ for all $t$.
\end{proof}

\paragraph{Coupled Learning}

The CL case requires a small modification. As established in Section~\ref{EP C Laws}, the continuous CL flow is \emph{not} gradient flow of any loss: differentiating the weighted loss $\Phi_D = \tfrac{1}{2}\langle r, Dr\rangle$ along the CL dynamics yields the cubic remainder $\tfrac{1}{2}\langle r, \dot D r\rangle$ in \eqref{eq:PhiD_dot}, which translates at the discrete level into a genuine non-gradient term in the update identity. Writing $y_D^t := \hat{L}^{-1}\hat{Q}D r^t$ and using $\partial_{k_e}\Phi_D = -(\Delta_e x_0)(\Delta_e y_D) + \tfrac{1}{2}(\Delta_e y_D)^2$, a calculation paralleling the EP derivation gives, edgewise,
\begin{equation}\label{eq:CL_discrete}
k_e^{t+1} - k_e^t = \tau\,(\Delta_e x_0^t)(\Delta_e y_D^t) + \tfrac{\tau\eta}{2}\,(\Delta_e y_D^t)^2,
\end{equation}
or equivalently in vector form, $k^{t+1} - k^t = -\tau\,\nabla\Phi_D(k^t) + \tfrac{\tau}{2}\,\varepsilon_{\rm CL}^t$ with $\varepsilon^t_{{\rm CL},e} := (1+\eta)(\Delta_e y_D^t)^2$. The key difference from the EP case is that the correction $\varepsilon_{\rm CL}^t$ is $O(\|r^t\|^2)$ rather than $O(\eta\|r^t\|^2)$: it survives the $\eta \to 0$ limit, reflecting the fact that CL is not exactly gradient flow of $\Phi_D$ even in continuous time. This non-gradient term is precisely what produced the cubic remainder $C_D \Phi_D^{3/2}$ in Theorem~\ref{thm:CL_coercivity}, and it requires the same smallness of the initial loss to control.

\begin{theorem}\label{thm:discrete_CL}
Let $k^* \in \mathcal{S}$ satisfy the coercivity condition of Theorem~\ref{thm:CL_coercivity}, and let $\lambda_D, c_D$ denote the coercivity constants for $\Phi_D$ at $k^*$ (i.e., $\|\nabla\Phi_D\|^2 \ge 2\lambda_D c_D^2\,\Phi_D$ on a neighborhood of $k^*$, from Lemma~\ref{lem:CL_coercivity_iff} and the bounds on $\lambda_{\min}(D)$, $\lambda_{\max}(D)$). There exist a neighborhood $U \ni k^*$, a Lipschitz constant $L_{\Phi_D} > 0$ for $\nabla\Phi_D$ on $U$, a constant $C > 0$ with $\|\Delta y(r)\| \le C\|r\|$ on $U$, and a radius $\delta > 0$ such that for any
\[
0 < \tau \le \tfrac{1}{L_{\Phi_D}}, \qquad \|k^0 - k^*\| < \delta, \qquad \Phi_D(k^0) \le \tfrac{\lambda_D c_D^2}{8(1+\eta)^2 C^4},
\]
the iterates of \eqref{eq:CL_discrete} remain in $U$ for all $t \ge 0$ and satisfy
\[
\Phi_D(k^{t+1}) \le \gamma_D\,\Phi_D(k^t), \qquad \gamma_D := 1 - \tfrac{\lambda_D c_D^2\,\tau}{4} \in (0,1).
\]
\end{theorem}

The proof follows the template of Theorem~\ref{thm:discrete_EP}. The only substantive change is in the bound on the error term: since the correction with components $\varepsilon^t_{{\rm CL},e} = (1+\eta)(\Delta_e y_D^t)^2$ is $O(\|r^t\|^2)$ independent of $\eta$, the smallness hypothesis on $\Phi_D(k^0)$ involves $(1+\eta)^2$ rather than $\eta^2$. In particular, unlike the EP case, the smallness is not relieved by taking $\eta$ small. This is the discrete analogue of the cubic $\Phi_D^{3/2}$ term in Theorem~\ref{thm:CL_coercivity}: the non-gradient part of CL must be controlled by starting near $\mathcal{S}$, not merely by making the nudge small.

\paragraph{Adjoint Coupled Learning}

As noted in Section~\ref{EP C Laws}, AL can be reinterpreted as EP with the active edge set defined according to the reference state $x_0$ rather than the free state. Theorem~\ref{thm:discrete_EP} therefore applies verbatim to AL, provided the constants $\lambda, c, C$ are computed with respect to this active edge set.
\section*{Funding}
XL and YM were partially supported by NSF through the University of Pennsylvania Materials Research Science and Engineering Center (MRSEC) (DMR-2309043).

\section*{Acknowledgments}
The authors would like to thank Han Zhou and Te-Sheng Lin for helpful discussions.

\section*{Disclosure of AI use}
AI-assisted tools (Claude, Anthropic) were used for editing, formatting, and literature search during the preparation of this manuscript. All mathematical content was developed by the authors.

\section*{Data availability}
No experimental data were used in this work. The Python scripts used to generate Figures~\ref{fig:coercivity_heatmap} and~\ref{fig:convergence_semilog} are available from the corresponding author upon request.

\end{document}